\documentclass[10pt]{article}
\usepackage{amsmath,amssymb,amsthm}

\title{Complexity growth of a typical triangular billiard is weakly exponential}
\author{Dmitri Scheglov\\
Fluminense Federal University}

\theoremstyle{plain}
\newtheorem{Lemma}{Lemma}[section]
\newtheorem{Theorem}{Theorem}[section]
\newtheorem{Corollary}{Corollary}[section]
\newtheorem{thm}{Theorem}
\newtheorem{conj}{Conjecture}
\begin{document}

\maketitle
\setlength{\parindent}{0pt}

\begin{abstract}
\noindent
We prove that for any $\epsilon>0$ the growth rate $P_n$ of generalized diagonals of a typical (in the Lebesgue measure sense) triangular billiard satisfies $P_n<Ce^{n^{\epsilon}}$. This makes further progress towards a solution of problem 3 in Katok's list ``Five most resistant problems in dynamics".
\end{abstract}

\section{Overview}

A generalized diagonal of a polygonal billiard is an orbit which connects two vertices. One of the characteristics of billiard dynamics is a complexity function $P_n$, which is a number of generalized diagonals of length no greater than $n$. Here by length we mean the number of reflections. There are also other versions of billiard complexity, such as position complexity, directional complexity and word complexity which are closely related to our definition( see [1], [2], [3], [4]).
\

\

Katok [7] proved the subexponential estimate.

\begin{thm}[Katok, 1987] For any polygon: $\lim\frac{\ln(P_n)}{n}=0$.
\end{thm}

Masur [9], [10] using Teichmuller theory proved quadratic estimates for any \textsl{rational-angled} polygon:
\

\begin{thm}[Masur, 1990] For any polygon with angles in $\pi\mathbb{Q}$  there are constants $C_1, C_2>0$ such that: $C_1\cdot n^2<P_n<C_2\cdot n^2$.
\end{thm}

Since the local geometry of polygonal billiards is similar and the result of Masur gives quadratic growth in the rational case, Katok formulated the following conjecture, which he included in his list ``Five most resistant problems in dynamics''[6]:
\

\begin{conj}[Katok] For any polygon and any $\epsilon>0$: $P_n<Cn^{2+\epsilon}$.
\end{conj}
\

However even though the growth of $P_n$ is conjecturally just more than quadratic, even explicit subexponential upper bound was an open problem for a long time, still not fully resolved. The key problem here is a luck of structure of irrational polygonal billiard as unlike the rational case it is not equivalent to the geodesic flow on the flat surface with singularities, so one can not use tools from Teichmuller theory.
 \

 \

 Recently we provided an explicit subexponential estimate for a full measure set of triangular billiards[11]:
 \begin{thm}[S, 2012] For a typical triangle and any $\epsilon>0$ there is a constant $C>0$ such that: $P_n<Ce^{n^{\sqrt{3}-1+\epsilon}}$.
\end{thm}

In the proof of theorem 3 we introduced a natural technic of indexed partitions for dealing with generalized diagonals. The key idea was to show that if there are too many generalized diagonals then by some kind of a pigeonhole principle there is a sequence of relatively short billiard trajectories which stay too close to one of the vertices at three different time moments. The last fact implies that some sequence of trigonometric polynomials must take very small values which is not possible on the set of full measure by the result of Kaloshin and Rodnianski[5].
\

\

In the current paper we again use indexed partitions as a convenient framework, however we will rely on quite different geometric ideas which require more delicate analysis of billiard dynamics. We will still need an abstract result by Kaloshin and Rodnianski about trigonometric polynomials however it will not play one of the main roles as it did in our previous paper. From this point of view we now use more dynamical arguments rather than abstract properties of trigonometric polynomials. The aim of the paper is to prove the following theorem:
\

\begin{thm}[Weakly exponential estimate] For a typical triangle and any $\epsilon>0$ there is a constant $C>0$ such that: $P_n<Ce^{n^{\epsilon}}$.
\end{thm}
\section{Interval partitions}

We consider a triangle, a fixed vertex and the corresponding angular segment located at the vertex, which we naturally associate with an interval $I\subseteq [0, \pi]$ using the angular distance on it. Points on the interval then correspond to rays emanating from the vertex. We introduce a useful reduced quantity $Q_n$ as a number of generalized diagonals emanating from the vertex of length no greater than $n$.
Now let us create a decreasing sequence of finite indexed partitions $\xi_n $ of $I$ on subintervals as follows. $\xi_0=I$ is a trivial partition with one element. Cutting points of partition $\xi_n$ are those corresponding to the generalized diagonals of length no greater than $n$.
Two observations immediately follow from this construction:
\

\

$1)$ Inside each interval of the partition $\xi_n$ there is at most one point of the partition $\xi_{n+1}$.
\

$2)$ The number $Q_n$ is exactly the number of cutting points of $\xi_n$ and each cutting point has an index, namely the length of the corresponding generalized diagonal.

\begin{Lemma} Let $\xi_n$, $\xi_{n+1}$, $\ldots$ , $\xi_{n+c}$ be a finite sequence of indexed partitions such that $Q_{n+c}>2Q_n+1$. Then there are at least
$Q_{n+c}-2Q_n-1$ intervals of the partition $\xi_{n+c}$ such that the indices of their endpoints belong to the interval $[n+1,n+c]$.
\end{Lemma}
\begin{proof}
Let us consider intervals of the partition $\xi_n$ and the cutting points of the partition $\xi_{n+c}$ inside these intervals. Let $X$ be the set of intervals of $\xi_n$ which have no points of $\xi_{n+c}$ inside, $Y$ be the set of intervals of $\xi_n$ which have 1 point of $\xi_{n+c}$ inside and $Z$ be the set of intervals which have 2 and more points of $\xi_{n+c}$ inside. Let us denote the cardinalities of these sets as $|X|=x$, $|Y|=y$, $|Z|=z$. We immediately have: $x+y+z=Q_n+1$.
\

\

Let us now consider a fixed interval $I_i\in Z$, where $i$ belongs to the index set, parameterizing $Z$. If $S_i$ is a number of points of $\xi_{n+c}$ inside $I_i$
then we have at least $N_i=S_i-1$ intervals of $\xi_{n+c}$ with indices in the interval $[n+1, n+c]$ located inside $I_i$. Then all the intervals from $Z$ contain at least $N= \sum N_i=\sum S_i -z$ intervals of $\xi_{n+c}$ with the required property.
\

\

Counting all the points of $\xi_{n+c}$ we have: $y+\sum S_i+ Q_n=Q_{n+c}$. As we noted that $x+y+z=Q_n+1$ then by subtracting we obtain:
$\sum S_i-z=Q_{n+c}-2Q_n-1+x$ so $N\geq Q_{n+c}-2Q_n-1$.
\end{proof}
\section{Combinatorial geometry of orbits}
\subsection{ Unfolding of a billiard trajectory}
We would like to remind of a useful unfolding construction associated to any polygonal billiard [8].
We fix a polygon on the plane and consider a time moment when a particular billiard orbit hits a polygon side. Then instead of reflecting the orbit we continue it as a straight line and reflect the polygon along the line.
As we continue this process indefinitely the sequence of polygons obtained this way is called unfolding of the polygon along the orbit.
The figure below illustrates the unfolding of a triangle along an orbit.
\

\unitlength 1mm 
\linethickness{0.4pt}
\ifx\plotpoint\undefined\newsavebox{\plotpoint}\fi 
\begin{picture}(55.066,43.929)(0,0)
\thicklines
\multiput(9.068,3.568)(.0337049505,.0599009901){505}{\line(0,1){.0599009901}}
\multiput(9.068,3.568)(.033717073,.133785366){205}{\line(0,1){.133785366}}
\multiput(9.068,39.542)(.033717073,-.041697561){205}{\line(0,-1){.041697561}}
\put(9.068,3.642){\line(0,1){35.974}}
\multiput(9.068,3.568)(-.033716495,.135242268){194}{\line(0,1){.135242268}}
\multiput(2.5,29.688)(.03365641,.050317949){195}{\line(0,1){.050317949}}
\multiput(15.977,31.008)(.12457799,.03373987){81}{\line(1,0){.12457799}}
\multiput(25.963,33.741)(.03344494,-.17040995){66}{\line(0,-1){.17040995}}
\multiput(28.17,22.494)(-.0344693836,-.0337118147){555}{\line(-1,0){.0344693836}}
\multiput(25.898,33.676)(.0381101812,-.0336974234){661}{\line(1,0){.0381101812}}
\multiput(51.089,11.402)(-.0691622931,.033646521){331}{\line(-1,0){.0691622931}}
\put(25.986,33.676){\line(1,0){11.049}}
\multiput(37.035,33.676)(.0337329529,-.0533749254){414}{\line(0,-1){.0533749254}}
\multiput(37.035,33.588)(.03339135,.11490554){90}{\line(0,1){.11490554}}
\multiput(40.04,43.929)(.0337237559,-.0995394732){325}{\line(0,-1){.0995394732}}

\put(55.066,25.014){\vector(1,0){.07}}\put(3.801,25.102){\line(1,0){51.2655}}
\end{picture}

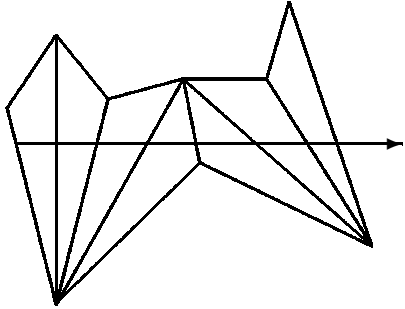
Fig 1. Triangle unfolding
\

\

For a given triangle the shape obtained from a triangle by reflection about one side is called a \textsl{kite}. It is clear that for any triangle unfolding there is an associated kite unfolding. We will use both unfoldings having in mind the natural correspondence between them.
The next figure illustrates the corresponding kite unfolding.
\

\unitlength 1mm 
\linethickness{0.4pt}
\ifx\plotpoint\undefined\newsavebox{\plotpoint}\fi 
\begin{picture}(55.066,43.929)(0,0)
\thicklines
\multiput(9.068,3.568)(.033717073,.133785366){205}{\line(0,1){.133785366}}
\multiput(9.068,39.542)(.033717073,-.041697561){205}{\line(0,-1){.041697561}}
\multiput(9.068,3.568)(-.033716495,.135242268){194}{\line(0,1){.135242268}}
\multiput(2.5,29.688)(.03365641,.050317949){195}{\line(0,1){.050317949}}
\multiput(15.977,31.008)(.12306098,.03332927){82}{\line(1,0){.12306098}}
\multiput(25.963,33.741)(.03343939,-.17040909){66}{\line(0,-1){.17040909}}
\multiput(28.17,22.494)(-.0344684685,-.0337117117){555}{\line(-1,0){.0344684685}}
\multiput(51.089,11.402)(-.069163142,.0336465257){331}{\line(-1,0){.069163142}}
\put(25.986,33.676){\line(1,0){11.049}}
\multiput(37.035,33.588)(.03338889,.1149){90}{\line(0,1){.1149}}
\put(55.066,25.014){\vector(1,0){.07}}\put(3.801,25.102){\line(1,0){51.265}}
\multiput(40.129,43.841)(.052612616,-.033672074){168}{\line(1,0){.052612616}}
\multiput(48.968,38.184)(.03367207,-.42510994){63}{\line(0,-1){.42510994}}
\multiput(37.035,33.676)(.033702355,-.0536270178){417}{\line(0,-1){.0536270178}}
\end{picture}

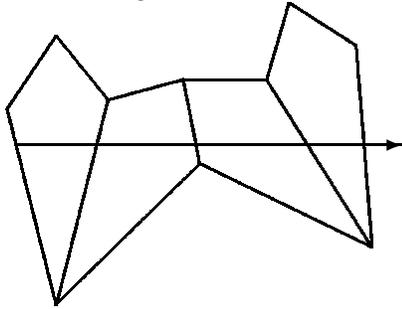
Fig 2. Kite unfolding
\

\

As we see from the figure above, any kite unfolding along the orbit consists of consecutive rotations of the kite along one of the two kite vertices, corresponding to the angles $\alpha$ and $\beta$ of the original triangle with angles correspondingly $ 2\alpha$ and $2\beta$.
\

We now assume that a kite is located in the standard Euclidean $xy$ coordinate plane and introduce several notations.
\

\

The \textsl{$\alpha$-vertex} and \textsl{$\beta$- vertex} are kite vertices corresponding to the angles $2\alpha$ and $2\beta$.
The two other vertices are called \textsl{side vertices}.
\

\

 The \textsl{kite diagonal} is a vector going from the $\alpha$-vertex to the $\beta$-vertex.
\

\

The \textsl{kite angle} is a \textsl{counterclockwise} angle between $x$-axis and the kite diagonal.
\

\

On Fig 3. $A$ and $B$ are $\alpha$ and $\beta$ vertices correspondingly, vector $\overrightarrow{AB}$ is a kite diagonal, $C$ and $D$ are side vertices.
Fig 4. shows a kite in standard position on the $xy$ plane.
\

\

 Note that any unfolding of $K$ is uniquely characterized by the sequence of angles $\pm 2\alpha$ or $\pm 2\beta$ depending on the kite vertex we rotate about and the direction of rotation. Such a sequence of angles is called the \textsl{combinatorics} of a kite unfolding.
\

\

\unitlength=1.00mm
\special{em:linewidth 0.4pt}
\linethickness{0.4pt}
\begin{picture}(112.33,67.34)
\put(2.00,19.00){\line(1,0){110.33}}
\put(9.34,4.67){\line(0,1){62.67}}
\put(62.34,29.34){\line(1,1){34.33}}
\put(96.67,63.67){\line(-1,0){18.33}}
\put(78.34,63.67){\line(-1,-2){17.67}}
\put(96.67,63.67){\line(-1,-1){36.00}}
\put(60.67,27.67){\line(1,1){36.00}}
\put(60.67,28.00){\line(5,2){36.00}}
\put(96.67,42.33){\line(0,1){21.33}}
\put(61.00,28.33){\circle*{1.33}}
\put(96.67,63.67){\circle*{2.11}}
\put(78.00,63.67){\circle*{2.00}}
\put(96.67,42.33){\circle*{2.00}}
\put(61.00,28.33){\circle*{2.00}}
\put(56.67,28.00){\makebox(0,0)[cc]{$A$}}
\put(100.33,63.67){\makebox(0,0)[cc]{$B$}}
\put(73.67,63.67){\makebox(0,0)[cc]{$C$}}
\put(100.33,42.33){\makebox(0,0)[cc]{$D$}}
\put(110.33,15.33){\makebox(0,0)[cc]{$x$}}
\put(6.00,66.00){\makebox(0,0)[cc]{$y$}}
\put(68.33,38.67){\makebox(0,0)[cc]{$\alpha$}}
\put(71.33,34.67){\makebox(0,0)[cc]{$\alpha$}}
\put(89.00,60.67){\makebox(0,0)[cc]{$\beta$}}
\put(94.00,56.67){\makebox(0,0)[cc]{$\beta$}}
\end{picture}
\

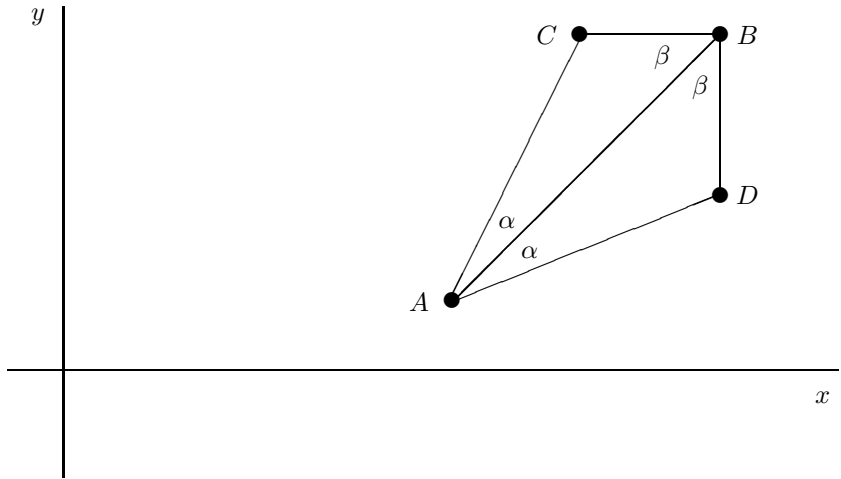
Fig 3. A kite on the $xy$ coordinate plane
\

\unitlength=0.5mm
\special{em:linewidth 0.4pt}
\linethickness{0.4pt}
\begin{picture}(129.33,114.67)
\put(9.00,52.00){\line(1,-1){29.67}}
\put(9.33,51.67){\line(1,1){29.33}}
\put(38.66,81.00){\line(1,-2){14.67}}
\put(53.33,51.67){\line(-1,-2){14.67}}
\put(9.33,52.00){\line(1,0){117.67}}
\put(9.33,12.67){\line(0,1){99.33}}
\put(15.00,54.00){\makebox(0,0)[cc]{$\alpha$}}
\put(15.00,49.67){\makebox(0,0)[cc]{$\alpha$}}
\put(49.00,54.33){\makebox(0,0)[cc]{$\beta$}}
\put(49.00,49.00){\makebox(0,0)[cc]{$\beta$}}
\put(129.33,47.33){\makebox(0,0)[cc]{$x$}}
\put(3.33,114.67){\makebox(0,0)[cc]{$y$}}
\end{picture}
\

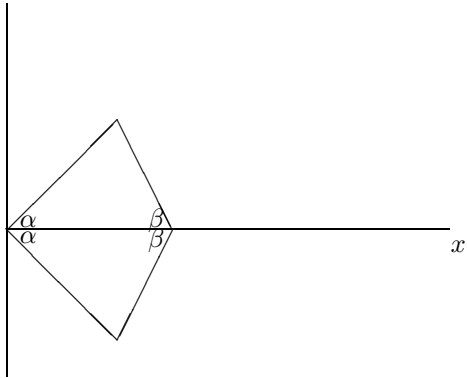
Fig 4. A kite in the standard position on the $xy$ plane

\begin{Lemma} Assume a kite $K$ with a diagonal length 1 is in standard position and a kite $K'$ is obtained from $K$ by means of a particular combinatorics of length $n$. Let $x^{\alpha}_n$, $y^{\alpha}_n$ and $x^{\beta}_n$, $y^{\beta}_n$ be the coordinates of $\alpha$ and $\beta$  vertices of $K'$ and let $x_n$, $y_n$ be the coordinates of either of the two side vertices of $K'$. Then:
\

$1)$ $x^{\alpha}_n$, $y^{\alpha}_n$, $x^{\beta}_n$, $y^{\beta}_n$ are represented by trigonometric polynomials of angles $\alpha$, $\beta$ with integer coefficients, depending only on the combinatorics and of degree at most $2n-2$.
\

$2)$ $x_n=P_{2n}(\alpha, \beta)+ \frac{\sin (\beta)}{\sin(\alpha+\beta)}\cdot \cos(m\alpha+l\beta)$
\

$y_n=Q_{2n}(\alpha, \beta)+ \frac{\sin (\beta)}{\sin(\alpha+\beta)}\cdot \sin(m\alpha+l\beta)$,
\

where $P_{2n}(\alpha,\beta)$, $Q_{2n}(\alpha, \beta)$ are trigonometric polynomials with integer coefficients of degree at most $2n-2$ and $|m|+|l|\leq 2n-1$.
\end{Lemma}

\begin{proof} The proof of the first statement goes by an easy induction on $n$. For n=1 the statement is trivial. If $\phi_n$ is the kite angle on the $n$-th step and on the $n+1$-th step we rotate, say, about $\alpha$-vertex, then $\phi_{n+1}=\phi_n\pm 2\alpha$ and $x^{\alpha}_{n+1}=x^{\alpha}_n$,  $y^{\alpha}_{n+1}=y^{\alpha}_n$,
$x^{\beta}_{n+1}=x^{\alpha}_n+\cos(\phi_{n+1})$,  $y^{\beta}_{n+1}=y^{\alpha}_{n}+\sin(\phi_{n+1})$.
The case when we rotate about $\beta$-vertex is entirely analogous. This completes the induction step.
\

\

The second statement easily follows from the first one by noticing that the length of the side, adjacent to the $\alpha$-vertex is $\frac{\sin (\beta)}{\sin(\alpha+\beta)}$ and so if $\phi_n$ is the kite angle of $K'$ then $x_n=x^{\alpha}_n+\frac{\sin (\beta)}{\sin(\alpha+\beta)}\cdot \cos(\phi_n\pm\alpha)$ and
 $y_n=y^{\alpha}_n+\frac{\sin (\beta)}{\sin(\alpha+\beta)}\cdot \sin(\phi_n\pm\alpha)$.
\end{proof}

\subsection{Local geometry near singular points}
In this section we study local dynamics of generalized diagonals near singular points. We consider a triangle with a side 1 and adjacent angles $\alpha$, $\beta$ in the standard position on the plane as shown on fig 5. and a generalized diagonal.
\

\

The \textsl{diagonal angle} is an oriented angle between $x$-axis and a diagonal. We want to emphasize that the generalized diagonal should not necessarily start from the `zero' vertex
of  a triangle in the standard position. On the figure below $\xi$ is a diagonal angle.
\

\unitlength .5mm 
\linethickness{0.4pt}
\ifx\plotpoint\undefined\newsavebox{\plotpoint}\fi 
\begin{picture}(236.75,93.339)(0,0)
\multiput(236,59)(.09375,-.0625){8}{\line(1,0){.09375}}
\put(142.836,76.721){\line(0,1){0}}
\put(85.206,41.012){\line(0,1){0}}
\thicklines
\put(28.285,21.213){\vector(1,0){125.159}}
\multiput(28.285,20.86)(.0673836634,.068259901){404}{\line(0,1){.068259901}}
\multiput(55.508,48.437)(.0962673267,-.0673861386){404}{\line(1,0){.0962673267}}
\multiput(163.343,92.985)(.117833,.059){6}{\line(1,0){.117833}}
\put(144.958,88.036){\line(0,-1){43.487}}
\multiput(144.958,44.548)(-.0673754717,.089390566){530}{\line(0,1){.089390566}}
\multiput(109.249,91.925)(.67375472,-.06671698){53}{\line(1,0){.67375472}}
\put(151.322,25.456){$x$}
\multiput(28.5,21)(.1165997994,.06745235707){997}{\line(1,0){.1165997994}}
\qbezier(50.75,33.75)(54.25,27.125)(53.75,21)
\put(56,27.75){$\xi$}
\end{picture}

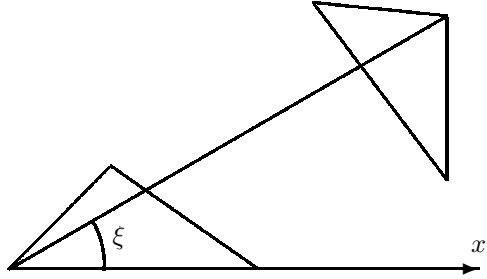
Fig 5. The diagonal angle
\

\

Now let us look more carefully on the behaviour of billiard trajectories near the end of the generalized diagonal. On the fig 6. we see an unfolding corresponding to all trajectories
emanating from the vertex in the \textsl{clockwise} small neighbourhood of the generalized diagonal. Of course an analogous picture and all the arguments below hold for a counterclockwise neighbourhood.
\

\

The unfolding about the end vertex of the diagonal is called \textsl{ a rose}. It represents the unfolding of all billiard trajectories emanating from the vertex $A$ and clockwise close to the diagonal.
\

\

On the fig 6. the dashed line is a formal geometric continuation of the generalized diagonal. The only triangle $BDE$ in the rose which it intersects is called \textsl{the exit triangle}. Notice that the side $BE$ of the exit triangle may lie on the formal continuation of the diagonal. This possibility causes no difficulties in our arguments.
\

\

The relative position of the triangle $BDE$ on the $xy$-plane up to parallel translations is called \textsl{the exit position} of the exit triangle corresponding to the diagonal $AB$. Notice that many different diagonals emanating from $A$ may have exit triangles in the same exit position.
\

\

The oriented angle $\theta$ between $x$-axis and the side $BD$ of the exit triangle is called \textsl{the exit angle}.
\

\unitlength .5mm 
\linethickness{0.4pt}
\ifx\plotpoint\undefined\newsavebox{\plotpoint}\fi 
\begin{picture}(280.021,129.71)(0,0)
\multiput(236,59)(.09375,-.0625){8}{\line(1,0){.09375}}
\put(142.836,76.721){\line(0,1){0}}
\put(85.206,41.012){\line(0,1){0}}
\thicklines
\put(28.285,21.213){\vector(1,0){125.159}}
\multiput(28.285,20.86)(.0673836634,.068259901){404}{\line(0,1){.068259901}}
\multiput(55.508,48.437)(.0962673267,-.0673861386){404}{\line(1,0){.0962673267}}
\multiput(163.343,92.985)(.117833,.059){6}{\line(1,0){.117833}}
\put(151.322,25.456){$x$}
\multiput(28.638,21.213)(.15817895772,.06744346116){1017}{\line(1,0){.15817895772}}
\multiput(189.506,89.803)(-.309361842,-.067453947){152}{\line(-1,0){.309361842}}
\multiput(142.483,79.55)(-.06733333,-.46579365){63}{\line(0,-1){.46579365}}
\put(138.241,50.205){\line(0,1){0}}
\put(138.241,50.205){\line(0,1){0}}
\multiput(138.241,50.205)(.0867325383,.0674582624){587}{\line(1,0){.0867325383}}
\multiput(189.153,89.803)(-.0674064748,-.1475251799){278}{\line(0,-1){.1475251799}}
\multiput(170.414,48.791)(-1.5320476,.0673333){21}{\line(-1,0){1.5320476}}
\multiput(170.493,48.772)(.084474886,-.067191781){219}{\line(1,0){.084474886}}
\put(188.993,34.057){\line(0,1){55.5}}
\multiput(188.993,89.556)(.0673345196,-.145886121){281}{\line(0,-1){.145886121}}
\multiput(207.914,48.562)(-.087023256,-.067465116){215}{\line(-1,0){.087023256}}
\multiput(207.914,48.562)(.30764634,.06665854){82}{\line(1,0){.30764634}}
\multiput(233.141,54.028)(-.0832981132,.0674320755){530}{\line(-1,0){.0832981132}}
\multiput(188.993,89.767)(.311617021,-.067099291){141}{\line(1,0){.311617021}}
\put(232.931,80.306){\line(0,-1){25.858}}
\multiput(232.931,80.096)(.067407609,.116538043){184}{\line(0,1){.116538043}}
\multiput(245.334,101.539)(-.315337079,-.067320225){178}{\line(-1,0){.315337079}}
\multiput(189.204,89.556)(.0890878553,.0673591731){387}{\line(1,0){.0890878553}}
\put(223.681,115.624){\line(3,-2){21.443}}
\multiput(188.852,89.205)(.150211,.066607){6}{\line(1,0){.150211}}
\multiput(190.655,90.005)(.150211,.066607){6}{\line(1,0){.150211}}
\multiput(192.457,90.804)(.150211,.066607){6}{\line(1,0){.150211}}
\multiput(194.26,91.603)(.150211,.066607){6}{\line(1,0){.150211}}
\multiput(196.063,92.403)(.150211,.066607){6}{\line(1,0){.150211}}
\multiput(197.865,93.202)(.150211,.066607){6}{\line(1,0){.150211}}
\multiput(199.668,94.001)(.150211,.066607){6}{\line(1,0){.150211}}
\multiput(201.47,94.8)(.150211,.066607){6}{\line(1,0){.150211}}
\multiput(203.273,95.6)(.150211,.066607){6}{\line(1,0){.150211}}
\multiput(205.075,96.399)(.150211,.066607){6}{\line(1,0){.150211}}
\multiput(206.878,97.198)(.150211,.066607){6}{\line(1,0){.150211}}
\multiput(208.68,97.998)(.150211,.066607){6}{\line(1,0){.150211}}
\multiput(210.483,98.797)(.150211,.066607){6}{\line(1,0){.150211}}
\multiput(212.285,99.596)(.150211,.066607){6}{\line(1,0){.150211}}
\multiput(214.088,100.395)(.150211,.066607){6}{\line(1,0){.150211}}
\multiput(215.89,101.195)(.150211,.066607){6}{\line(1,0){.150211}}
\multiput(217.693,101.994)(.150211,.066607){6}{\line(1,0){.150211}}
\multiput(219.496,102.793)(.150211,.066607){6}{\line(1,0){.150211}}
\multiput(221.298,103.593)(.150211,.066607){6}{\line(1,0){.150211}}
\multiput(223.101,104.392)(.150211,.066607){6}{\line(1,0){.150211}}
\multiput(224.903,105.191)(.150211,.066607){6}{\line(1,0){.150211}}
\multiput(226.706,105.99)(.150211,.066607){6}{\line(1,0){.150211}}
\multiput(228.508,106.79)(.150211,.066607){6}{\line(1,0){.150211}}
\multiput(230.311,107.589)(.150211,.066607){6}{\line(1,0){.150211}}
\multiput(232.113,108.388)(.150211,.066607){6}{\line(1,0){.150211}}
\multiput(233.916,109.188)(.150211,.066607){6}{\line(1,0){.150211}}
\multiput(235.718,109.987)(.150211,.066607){6}{\line(1,0){.150211}}
\multiput(237.521,110.786)(.150211,.066607){6}{\line(1,0){.150211}}
\multiput(239.323,111.585)(.150211,.066607){6}{\line(1,0){.150211}}
\multiput(241.126,112.385)(.150211,.066607){6}{\line(1,0){.150211}}
\multiput(242.928,113.184)(.150211,.066607){6}{\line(1,0){.150211}}
\multiput(244.731,113.983)(.150211,.066607){6}{\line(1,0){.150211}}
\multiput(246.534,114.783)(.150211,.066607){6}{\line(1,0){.150211}}
\multiput(248.336,115.582)(.150211,.066607){6}{\line(1,0){.150211}}
\multiput(250.139,116.381)(.150211,.066607){6}{\line(1,0){.150211}}
\multiput(251.941,117.18)(.150211,.066607){6}{\line(1,0){.150211}}
\multiput(253.744,117.98)(.150211,.066607){6}{\line(1,0){.150211}}
\multiput(255.546,118.779)(.150211,.066607){6}{\line(1,0){.150211}}
\multiput(257.349,119.578)(.150211,.066607){6}{\line(1,0){.150211}}
\multiput(259.151,120.378)(.150211,.066607){6}{\line(1,0){.150211}}
\multiput(260.954,121.177)(.150211,.066607){6}{\line(1,0){.150211}}
\multiput(262.756,121.976)(.150211,.066607){6}{\line(1,0){.150211}}
\multiput(264.559,122.775)(.150211,.066607){6}{\line(1,0){.150211}}
\multiput(266.361,123.575)(.150211,.066607){6}{\line(1,0){.150211}}
\multiput(268.164,124.374)(.150211,.066607){6}{\line(1,0){.150211}}
\multiput(269.966,125.173)(.150211,.066607){6}{\line(1,0){.150211}}
\multiput(271.769,125.973)(.150211,.066607){6}{\line(1,0){.150211}}
\multiput(273.572,126.772)(.150211,.066607){6}{\line(1,0){.150211}}
\multiput(275.374,127.571)(.150211,.066607){6}{\line(1,0){.150211}}
\multiput(277.177,128.37)(.150211,.066607){6}{\line(1,0){.150211}}
\multiput(278.979,129.17)(.150211,.066607){6}{\line(1,0){.150211}}
\put(186.26,91.869){$B$}
\put(248.067,100.909){$D$}
\put(21.653,21.443){$A$}
\put(131.391,49.613){$C$}
\put(188.993,89.346){\vector(1,0){84.721}}
\put(223.89,118.357){$E$}
\put(52.767,31.324){\vector(-2,1){.141}}\qbezier(55.499,21.023)(60.65,28.065)(52.767,31.324)
\put(219.896,95.863){\vector(-1,1){.141}}\qbezier(220.316,89.136)(223.47,93.761)(219.896,95.863)
\put(60.25,25){$\xi$}
\put(225,91.5){$\theta$}
\end{picture}

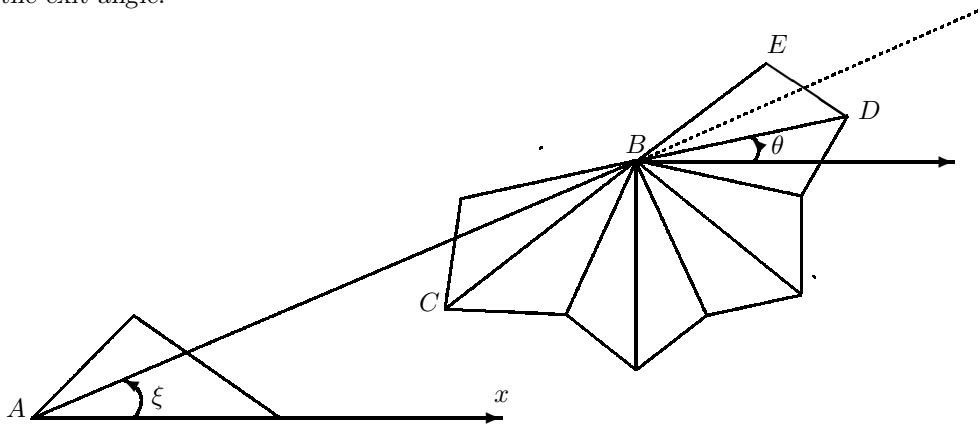
Fig 6. The rose
\

\

Now we are going to use just defined concept of a rose to analyze the local behaviour of the diagonals. For any $\delta>0$ small enough let $\Delta_{\delta}$ be the set of triangles with a fixed side of length 1 such that all angles are greater then $\delta$. From now and further we assume that we consider triangles from the set $\Delta_{\delta}$ for some $\delta>0$. As $\delta$ can be chosen arbitrarily small our arguments would ultimately work for the whole set of triangles. Let us now prove a technical statement that for two close enough diagonals the greater algebraic length implies the strictly greater geometric length.

\begin{Lemma} Let $\xi_n$ be an indexed partition and $I\in\xi_n$ be an interval with endpoints $x$, $y$ with indices $p<q$. There are universal constants $b$, $r>0$ such that if $L_p$ and $L_q$ are geometric lengths of the diagonals corresponding to points $x$, $y$ and the length $|I|< b/p$ then $L_q>L_p+r$.
\end{Lemma}
\begin{proof}
It is well known that there is a constant $D_{\delta}$ such that for any triangle from $\Delta_{\delta}$ and any diagonal of algebraic length $p$ its geometric length is
 less than $D_{\delta}p$. On fig 7. $AP$ is a diagonal corresponding to $x$ and $\phi$ is a circular segment of the angular beam, corresponding to $I$ on the distance $L_p+r$ from $A$ where $r$ is a small constant to be chosen later.
 \

 \

  The point $B$ lies outside the $\phi$-segment because $p<q$ and so all the trajectories from the $\phi$-segment intersect $BC$ as by definition of the partition $\xi_n$ they all have the same unfolding combinatorics of length at least $q$. The side lengths of all triangles from $\Delta_{\delta}$ are bounded from below by some constant $R_{\delta}$.
  \

  \

  Assuming that $|I|<b/p$ we have $|\phi|$=$|I|(L_p+r)\leq\frac{b}{p}(D_{\delta}p+r)=bD_{\delta}+br/p$ and it is now clear that there are small enough constants $b$, $r>0$ depending only on $\delta$ such that $|\phi|+r<R_{\delta}$. It implies that the points of the rose can not lie inside the $\phi$- segment because on the one hand they are separated from the $\phi$-segment by the point $B$ and on the other hand their distances to $P$ are bounded from below by $R_{\delta}$.
 \

 \unitlength .5mm 
\linethickness{0.4pt}
\ifx\plotpoint\undefined\newsavebox{\plotpoint}\fi 
\begin{picture}(280.021,129.71)(0,0)
\multiput(236,59)(.09375,-.0625){8}{\line(1,0){.09375}}
\put(142.836,76.721){\line(0,1){0}}
\put(85.206,41.012){\line(0,1){0}}
\thicklines
\put(28.285,21.213){\vector(1,0){125.159}}
\multiput(28.285,20.86)(.0673836634,.068259901){404}{\line(0,1){.068259901}}
\multiput(55.508,48.437)(.0962673267,-.0673861386){404}{\line(1,0){.0962673267}}
\multiput(163.343,92.985)(.117833,.059){6}{\line(1,0){.117833}}
\put(151.322,25.456){$x$}
\multiput(28.638,21.213)(.15817895772,.06744346116){1017}{\line(1,0){.15817895772}}
\multiput(189.506,89.803)(-.309361842,-.067453947){152}{\line(-1,0){.309361842}}
\multiput(142.483,79.55)(-.06733333,-.46579365){63}{\line(0,-1){.46579365}}
\put(138.241,50.205){\line(0,1){0}}
\put(138.241,50.205){\line(0,1){0}}
\multiput(138.241,50.205)(.0867325383,.0674582624){587}{\line(1,0){.0867325383}}
\multiput(189.153,89.803)(-.0674064748,-.1475251799){278}{\line(0,-1){.1475251799}}
\multiput(170.414,48.791)(-1.5320476,.0673333){21}{\line(-1,0){1.5320476}}
\multiput(170.493,48.772)(.084474886,-.067191781){219}{\line(1,0){.084474886}}
\put(188.993,34.057){\line(0,1){55.5}}
\multiput(188.993,89.556)(.0673345196,-.145886121){281}{\line(0,-1){.145886121}}
\multiput(207.914,48.562)(-.087023256,-.067465116){215}{\line(-1,0){.087023256}}
\multiput(207.914,48.562)(.30764634,.06665854){82}{\line(1,0){.30764634}}
\multiput(233.141,54.028)(-.0832981132,.0674320755){530}{\line(-1,0){.0832981132}}
\multiput(188.993,89.767)(.311617021,-.067099291){141}{\line(1,0){.311617021}}
\put(232.931,80.306){\line(0,-1){25.858}}
\multiput(232.931,80.096)(.067407609,.116538043){184}{\line(0,1){.116538043}}
\multiput(245.334,101.539)(-.315337079,-.067320225){178}{\line(-1,0){.315337079}}
\multiput(189.204,89.556)(.0890878553,.0673591731){387}{\line(1,0){.0890878553}}
\put(223.681,115.624){\line(3,-2){21.443}}
\multiput(188.852,89.205)(.150211,.066607){6}{\line(1,0){.150211}}
\multiput(190.655,90.005)(.150211,.066607){6}{\line(1,0){.150211}}
\multiput(192.457,90.804)(.150211,.066607){6}{\line(1,0){.150211}}
\multiput(194.26,91.603)(.150211,.066607){6}{\line(1,0){.150211}}
\multiput(196.063,92.403)(.150211,.066607){6}{\line(1,0){.150211}}
\multiput(197.865,93.202)(.150211,.066607){6}{\line(1,0){.150211}}
\multiput(199.668,94.001)(.150211,.066607){6}{\line(1,0){.150211}}
\multiput(201.47,94.8)(.150211,.066607){6}{\line(1,0){.150211}}
\multiput(203.273,95.6)(.150211,.066607){6}{\line(1,0){.150211}}
\multiput(205.075,96.399)(.150211,.066607){6}{\line(1,0){.150211}}
\multiput(206.878,97.198)(.150211,.066607){6}{\line(1,0){.150211}}
\multiput(208.68,97.998)(.150211,.066607){6}{\line(1,0){.150211}}
\multiput(210.483,98.797)(.150211,.066607){6}{\line(1,0){.150211}}
\multiput(212.285,99.596)(.150211,.066607){6}{\line(1,0){.150211}}
\multiput(214.088,100.395)(.150211,.066607){6}{\line(1,0){.150211}}
\multiput(215.89,101.195)(.150211,.066607){6}{\line(1,0){.150211}}
\multiput(217.693,101.994)(.150211,.066607){6}{\line(1,0){.150211}}
\multiput(219.496,102.793)(.150211,.066607){6}{\line(1,0){.150211}}
\multiput(221.298,103.593)(.150211,.066607){6}{\line(1,0){.150211}}
\multiput(223.101,104.392)(.150211,.066607){6}{\line(1,0){.150211}}
\multiput(224.903,105.191)(.150211,.066607){6}{\line(1,0){.150211}}
\multiput(226.706,105.99)(.150211,.066607){6}{\line(1,0){.150211}}
\multiput(228.508,106.79)(.150211,.066607){6}{\line(1,0){.150211}}
\multiput(230.311,107.589)(.150211,.066607){6}{\line(1,0){.150211}}
\multiput(232.113,108.388)(.150211,.066607){6}{\line(1,0){.150211}}
\multiput(233.916,109.188)(.150211,.066607){6}{\line(1,0){.150211}}
\multiput(235.718,109.987)(.150211,.066607){6}{\line(1,0){.150211}}
\multiput(237.521,110.786)(.150211,.066607){6}{\line(1,0){.150211}}
\multiput(239.323,111.585)(.150211,.066607){6}{\line(1,0){.150211}}
\multiput(241.126,112.385)(.150211,.066607){6}{\line(1,0){.150211}}
\multiput(242.928,113.184)(.150211,.066607){6}{\line(1,0){.150211}}
\multiput(244.731,113.983)(.150211,.066607){6}{\line(1,0){.150211}}
\multiput(246.534,114.783)(.150211,.066607){6}{\line(1,0){.150211}}
\multiput(248.336,115.582)(.150211,.066607){6}{\line(1,0){.150211}}
\multiput(250.139,116.381)(.150211,.066607){6}{\line(1,0){.150211}}
\multiput(251.941,117.18)(.150211,.066607){6}{\line(1,0){.150211}}
\multiput(253.744,117.98)(.150211,.066607){6}{\line(1,0){.150211}}
\multiput(255.546,118.779)(.150211,.066607){6}{\line(1,0){.150211}}
\multiput(257.349,119.578)(.150211,.066607){6}{\line(1,0){.150211}}
\multiput(259.151,120.378)(.150211,.066607){6}{\line(1,0){.150211}}
\multiput(260.954,121.177)(.150211,.066607){6}{\line(1,0){.150211}}
\multiput(262.756,121.976)(.150211,.066607){6}{\line(1,0){.150211}}
\multiput(264.559,122.775)(.150211,.066607){6}{\line(1,0){.150211}}
\multiput(266.361,123.575)(.150211,.066607){6}{\line(1,0){.150211}}
\multiput(268.164,124.374)(.150211,.066607){6}{\line(1,0){.150211}}
\multiput(269.966,125.173)(.150211,.066607){6}{\line(1,0){.150211}}
\multiput(271.769,125.973)(.150211,.066607){6}{\line(1,0){.150211}}
\multiput(273.572,126.772)(.150211,.066607){6}{\line(1,0){.150211}}
\multiput(275.374,127.571)(.150211,.066607){6}{\line(1,0){.150211}}
\multiput(277.177,128.37)(.150211,.066607){6}{\line(1,0){.150211}}
\multiput(278.979,129.17)(.150211,.066607){6}{\line(1,0){.150211}}
\put(21.653,21.443){$A$}
\put(133.073,42.255){$B$}
\multiput(28.801,21.233)(.18439360639,.06741458541){1001}{\line(1,0){.18439360639}}
\qbezier(213.379,88.925)(211.382,96.914)(208.964,98.175)
\qbezier(177.01,62.647)(211.802,48.877)(221.788,82.619)
\qbezier(220.106,103.01)(224.31,92.184)(221.788,82.619)
\put(214.009,91.448){$\phi$}
\qbezier(177.01,62.647)(165.763,67.693)(163.345,78.624)
\put(137.908,81.778){$C$}
\put(186.049,92.92){$P$}
\end{picture}
\

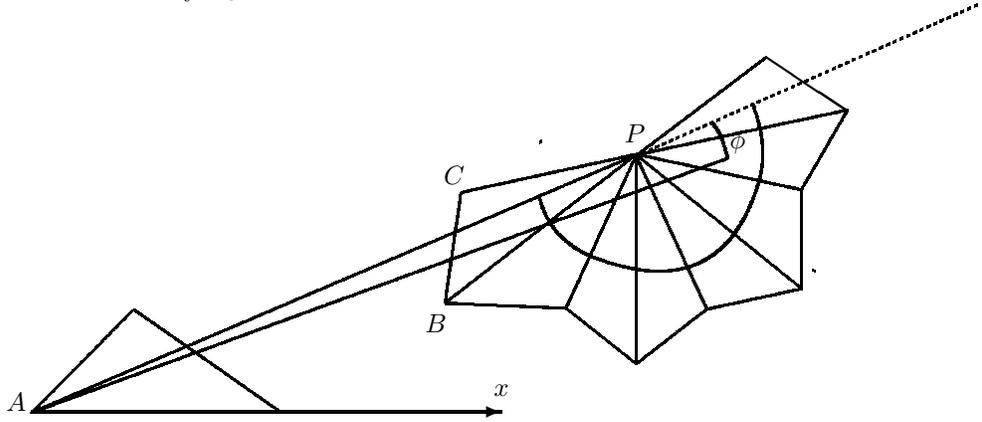
Fig 7. The $\phi$ - segment
\end{proof}
 Our next lemma proves a simple but important observation that for two close enough diagonals the segment connecting their endpoints is also a diagonal.
\begin{Lemma} Let $\xi_n$ be an indexed partition and $I\in\xi_n$ be an interval with endpoints with indices $p<q$. There is a universal constant $b$ such that if $|I|< b/q$ then the segment connecting the endpoints of the diagonals on fig 8. is either a side of the triangle or an unfolding of some generalized diagonal of algebraic length bounded by $q-p$.
\end{Lemma}

\begin{proof} We take $b$ from lemma 3.2. Then by lemma 3.2 the whole open angular $\phi$ - segment on fig 8. consists of orbits with the same unfolding of algebraic length at least $q$. This happens because any combinatorics change inside the $\phi$-segment must happen at the algebraic time greater than $q$, which by lemma 3.2 geometrically happens outside the $\phi$-segment of radius $L_q+r$.
As the combinatorics does not change then any straight line segment located strictly inside the angular $\phi$-segment does not hit any vertex and so either represents the unfolding of some billiard trajectory or belongs to the side of one of the unfolding triangles. In particular the segment $PQ$ represents the unfolding of some generalized diagonal or coincides with a side of the triangle. In the first case its algebraic length is obviously bounded above by $q-p$.
\

\unitlength .5mm 
\linethickness{0.4pt}
\ifx\plotpoint\undefined\newsavebox{\plotpoint}\fi 
\begin{picture}(243.231,109.107)(0,0)
\multiput(168.648,42.183)(.0875,-.058333){6}{\line(1,0){.0875}}
\put(103.433,54.588){\line(0,1){0}}
\put(63.092,29.591){\line(0,1){0}}
\thicklines
\put(23.247,15.732){\vector(1,0){87.611}}
\multiput(23.247,15.485)(.0673360424,.0682116608){283}{\line(0,1){.0682116608}}
\multiput(42.303,34.789)(.0961992933,-.0673385159){283}{\line(1,0){.0961992933}}
\multiput(117.788,65.972)(.12373,.06195){4}{\line(1,0){.12373}}
\put(109.373,18.702){$x$}
\multiput(23.494,15.732)(.15815674157,.06743398876){712}{\line(1,0){.15815674157}}
\put(100.216,36.026){\line(0,1){0}}
\put(100.216,36.026){\line(0,1){0}}
\put(18.605,15.893){$A$}
\multiput(135.875,63.768)(.128471,.059786){7}{\line(1,0){.128471}}
\multiput(137.674,64.605)(.128471,.059786){7}{\line(1,0){.128471}}
\multiput(139.472,65.442)(.128471,.059786){7}{\line(1,0){.128471}}
\multiput(141.271,66.279)(.128471,.059786){7}{\line(1,0){.128471}}
\multiput(143.07,67.116)(.128471,.059786){7}{\line(1,0){.128471}}
\multiput(144.868,67.953)(.128471,.059786){7}{\line(1,0){.128471}}
\multiput(146.667,68.79)(.128471,.059786){7}{\line(1,0){.128471}}
\multiput(148.465,69.627)(.128471,.059786){7}{\line(1,0){.128471}}
\multiput(150.264,70.464)(.128471,.059786){7}{\line(1,0){.128471}}
\multiput(152.063,71.301)(.128471,.059786){7}{\line(1,0){.128471}}
\multiput(153.861,72.138)(.128471,.059786){7}{\line(1,0){.128471}}
\multiput(155.66,72.975)(.128471,.059786){7}{\line(1,0){.128471}}
\multiput(157.458,73.812)(.128471,.059786){7}{\line(1,0){.128471}}
\multiput(159.257,74.649)(.128471,.059786){7}{\line(1,0){.128471}}
\multiput(161.056,75.486)(.128471,.059786){7}{\line(1,0){.128471}}
\multiput(162.854,76.323)(.128471,.059786){7}{\line(1,0){.128471}}
\multiput(164.653,77.16)(.128471,.059786){7}{\line(1,0){.128471}}
\multiput(166.451,77.997)(.128471,.059786){7}{\line(1,0){.128471}}
\multiput(168.25,78.834)(.128471,.059786){7}{\line(1,0){.128471}}
\multiput(170.049,79.671)(.128471,.059786){7}{\line(1,0){.128471}}
\multiput(171.847,80.508)(.128471,.059786){7}{\line(1,0){.128471}}
\multiput(173.646,81.345)(.128471,.059786){7}{\line(1,0){.128471}}
\multiput(175.444,82.182)(.128471,.059786){7}{\line(1,0){.128471}}
\multiput(177.243,83.019)(.128471,.059786){7}{\line(1,0){.128471}}
\multiput(179.041,83.856)(.128471,.059786){7}{\line(1,0){.128471}}
\multiput(180.84,84.693)(.128471,.059786){7}{\line(1,0){.128471}}
\multiput(182.639,85.53)(.128471,.059786){7}{\line(1,0){.128471}}
\multiput(184.437,86.367)(.128471,.059786){7}{\line(1,0){.128471}}
\multiput(186.236,87.204)(.128471,.059786){7}{\line(1,0){.128471}}
\multiput(188.034,88.041)(.128471,.059786){7}{\line(1,0){.128471}}
\multiput(189.833,88.878)(.128471,.059786){7}{\line(1,0){.128471}}
\multiput(191.632,89.715)(.128471,.059786){7}{\line(1,0){.128471}}
\multiput(193.43,90.552)(.128471,.059786){7}{\line(1,0){.128471}}
\multiput(195.229,91.389)(.128471,.059786){7}{\line(1,0){.128471}}
\multiput(197.027,92.226)(.128471,.059786){7}{\line(1,0){.128471}}
\multiput(198.826,93.063)(.128471,.059786){7}{\line(1,0){.128471}}
\multiput(200.625,93.9)(.128471,.059786){7}{\line(1,0){.128471}}
\multiput(202.423,94.737)(.128471,.059786){7}{\line(1,0){.128471}}
\multiput(204.222,95.574)(.128471,.059786){7}{\line(1,0){.128471}}
\multiput(206.02,96.411)(.128471,.059786){7}{\line(1,0){.128471}}
\multiput(207.819,97.248)(.128471,.059786){7}{\line(1,0){.128471}}
\multiput(209.618,98.085)(.128471,.059786){7}{\line(1,0){.128471}}
\multiput(211.416,98.922)(.128471,.059786){7}{\line(1,0){.128471}}
\multiput(213.215,99.759)(.128471,.059786){7}{\line(1,0){.128471}}
\multiput(215.013,100.596)(.128471,.059786){7}{\line(1,0){.128471}}
\multiput(216.812,101.433)(.128471,.059786){7}{\line(1,0){.128471}}
\multiput(218.611,102.27)(.128471,.059786){7}{\line(1,0){.128471}}
\multiput(220.409,103.107)(.128471,.059786){7}{\line(1,0){.128471}}
\multiput(222.208,103.944)(.128471,.059786){7}{\line(1,0){.128471}}
\multiput(224.006,104.781)(.128471,.059786){7}{\line(1,0){.128471}}
\multiput(225.805,105.618)(.128471,.059786){7}{\line(1,0){.128471}}
\multiput(227.604,106.455)(.128471,.059786){7}{\line(1,0){.128471}}
\multiput(229.402,107.292)(.128471,.059786){7}{\line(1,0){.128471}}
\multiput(231.201,108.129)(.128471,.059786){7}{\line(1,0){.128471}}
\multiput(23.545,15.557)(.20815621433,.06745420067){1016}{\line(1,0){.20815621433}}
\multiput(136.646,64.119)(.3350690917,.067444314){293}{\line(1,0){.3350690917}}
\multiput(234.612,83.88)(-.06722943,-.105646248){197}{\line(0,-1){.105646248}}
\multiput(221.367,63.068)(-.06688989,.51361878){88}{\line(0,1){.51361878}}
\multiput(215.481,108.266)(.0673319663,-.0860352903){281}{\line(0,-1){.0860352903}}
\multiput(234.471,83.739)(.172385,.054659){5}{\line(1,0){.172385}}
\multiput(236.195,84.286)(.172385,.054659){5}{\line(1,0){.172385}}
\multiput(237.919,84.833)(.172385,.054659){5}{\line(1,0){.172385}}
\multiput(239.642,85.379)(.172385,.054659){5}{\line(1,0){.172385}}
\multiput(241.366,85.926)(.172385,.054659){5}{\line(1,0){.172385}}
\qbezier(240.288,85.772)(237.765,94.812)(231.038,108.056)
\multiput(137.067,64.119)(.067365033,-.153313523){181}{\line(0,-1){.153313523}}
\multiput(149.26,36.369)(-.1048198878,.067342386){359}{\line(-1,0){.1048198878}}
\multiput(111.63,60.545)(.44257971,.06638696){57}{\line(1,0){.44257971}}
\multiput(136.857,64.329)(.0672412002,.0680140876){272}{\line(0,1){.0680140876}}
\multiput(155.146,82.829)(-.06688989,-.52317448){88}{\line(0,-1){.52317448}}
\put(132.232,67.482){$P$}
\put(235.873,75.891){$Q$}
\put(238.185,97.965){$\phi$}
\end{picture}
\

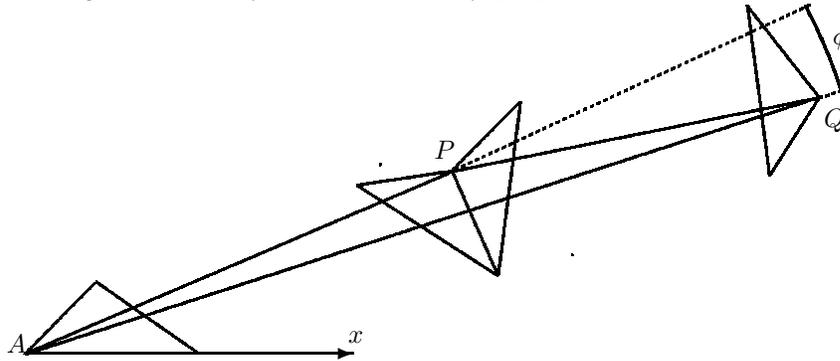
Fig 8. Connecting diagonal
\end{proof}
The following technical Lemma will be used in our later partition estimates.
\begin{Lemma}
Assume that there are fixed constants $r$, $D>0$ and a triangle $ABC$ $($Fig 9.$)$ such that $|AB|<|AC|$ and $|AC|<Dn$ for some positive integer $n$ and $|BC|>r$.
Then there are constants $K$, $b>0$ depending only on $r$, $D$ such that if $\phi<b/n$ then $\psi<K\phi n$.
\end{Lemma}
\

\unitlength .5mm 
\linethickness{0.4pt}
\ifx\plotpoint\undefined\newsavebox{\plotpoint}\fi 
\begin{picture}(220.947,71.477)(0,0)
\thicklines
\put(52.136,38.471){\line(1,0){168.811}}
\multiput(52.136,38.261)(.3257853474,.0673744567){493}{\line(1,0){.3257853474}}
\multiput(212.748,71.477)(-.1231319975,-.067357922){490}{\line(-1,0){.1231319975}}
\qbezier(175.118,50.664)(178.271,45.829)(178.061,38.471)
\qbezier(98.806,47.721)(101.959,43.201)(99.226,38.261)
\put(103.641,40.363){$\phi$}
\put(180.163,43.096){$\psi$}
\put(45.198,40.363){$A$}
\put(149.05,30.903){$B$}
\put(212.958,63.909){$C$}
\end{picture}
\

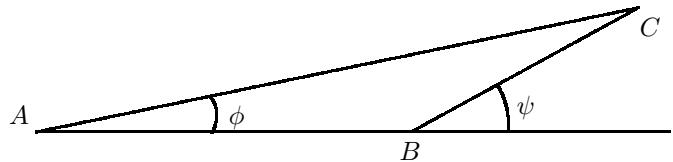
Fig 9. Triangle $ABC$
\begin{proof}
By the sinus theorem $\sin(\psi)=|AC|\sin(\phi)/|BC|$ and since $\sin(x)\approx x$ as $x\rightarrow 0$ then for small enough $b$ the conclusion follows.
\end{proof}
\section{Complexity estimate}
In this section we use the tools described above to get a global complexity estimate. At first we need an important lemma which says that for typical triangles from $\Delta_{\delta}$ two diagonals can not be too close to each other.

\begin{Lemma} There is a full measure set of triangles $X\subseteq\Delta_{\delta}$ such that for any triangle from $X$ there is a constant $a>0$ with the following property:
\

\

If $I\in\xi_n$ is an interval of the partition $\xi_n$ corresponding to any vertex of the triangle, then $|I|>e^{-an^2}$.

\end{Lemma}

\begin{proof}
We consider some triangle $\Delta\in\Delta_{\delta}$ in the standard position on the $xy$ - plane with angles $\alpha$, $\beta$ adjacent to the fixed side of length 1 and two diagonals from $\xi_n$ corresponding to some vertex. Let us assume that $A$ is the vertex of the original triangle from which two given diagonals start and $B$, $C$ are the end points of the diagonal unfoldings on the $xy$ - plane.
By Lemma 3.1 the $x$ and $y$ coordinates of points $A$, $B$, $C$ can be represented as:
\

\

  $x=\frac{P_{2n}(\alpha, \beta)}{\sin(\alpha+\beta)}$, $y=\frac{Q_{2n}(\alpha, \beta)}{\sin(\alpha+\beta)}$,

\

where $P$, $Q$ are trigonometric polynomials with integer coefficients of degree at most $2n$, so the area $S$ of the triangle $ABC$ can be represented as:
 \

 \

 $S=\frac{M(\alpha, \beta)}{\sin^2(\alpha+\beta)}$,
 \

 \

 where $M(\alpha, \beta)$ is a trigonometric polynomial with integer coefficients of degree at most $4n$. As the diagonals are different $M(\alpha,\beta)\neq 0$.
\

\

Let $\phi$ be an angle between the diagonals $AB$ and $AC$, then $S=|AB||AC|\sin(\phi)/2$. This implies that $\phi>\sin(\phi)=\frac{2S}{|AB||AC|}>lM(\alpha, \beta)/n^2$ for some constant $l>0$, because $|AB|$, $|AC|<D_{\delta}n$.
 We now refer to the theorem by Kaloshin and Rodnianski [5]  which can be formulated as follows:

\begin{Theorem}[Kaloshin, Rodnianski] There exist universal constants $R, h>0$ such that any non-zero trigonometric polynomial with integer coefficients $P$ in variables $\alpha, \beta, \gamma\in [0, 2\pi]$  of degree at most $m$ satisfies :
\

\

$Leb\lbrace (\alpha, \beta, \gamma): |P(\alpha, \beta, \gamma)|<e^{-Rm^{2}}\rbrace<e^{-hm}$.
\end{Theorem}

Any trigonometric polynomial $P(\alpha, \beta)$ in 2 variables can be considered as a polynomial $P(\alpha, \beta, \gamma)$ of three variables of the same degree, where the variable $\gamma$ is not present.
Moreover any level set for $P$ in variables $\alpha, \beta, \gamma$ is obtained from the level set for $P$ in variables $\alpha, \beta$ by multiplying on segment $[0, 2\pi]$ in variable $\gamma$. Then an easy use of the Fubini's theorem implies the following corollary:
\

\begin{Corollary} There exist universal constants $R, h>0$ such that any non-zero trigonometric polynomial with integer coefficients $P$ in variables $\alpha, \beta\in [0, 2\pi]$  of degree at most $m$ satisfies the following inequality:
\

\

$Leb\lbrace (\alpha, \beta): |P(\alpha, \beta)|<e^{-Rm^{2}}\rbrace<e^{-hm}$.
\end{Corollary}

Let $\mathcal{F}_n$ be a set of all trigonometric polynomials $M$ as above. Any polynomial is uniquely determined by the combinatorics and a choice of vertices, so the cardinality  $|\mathcal{F}_n|< e^{tn}$, for some $t>0$.
\

\

We now take $m=Fn$ where $F>4$ is a constant to be chosen later and pick some polynomial $M\in\mathcal{F}_n$. As the degree of $M$ is less than $m$ then by corollary 4.1 $Leb\lbrace (\alpha, \beta): |M(\alpha, \beta)|<e^{-RF^2n^{2}}\rbrace<e^{-hFn}$.
\

\

Now take the set of angles
${\mathcal{B}}_n=\lbrace (\alpha, \beta)| \exists M\in{\mathcal{F}}_n :|M(\alpha,\beta)|<e^{-RF^2n^2}\rbrace $.
\

\

As $|\mathcal{F}_n|<e^{tn}$ then $Leb({\mathcal{B}}_n)<e^{(t-hF)n}$. Picking $F>t/h$ we have that $\sum Leb({\mathcal{B}}_n)<\infty$ so by Borel-Cantelly argument for almost any pair
$(\alpha, \beta)$ and all large enough $n$ for any polynomial $M\in\mathcal{F}_n$ we have $|M(\alpha,\beta)|\geq e^{-RF^2n^2}$. Picking large enough $a> RF^2$( depending on $(\alpha,\beta)$) we complete the proof.
\end{proof}
For the rest of the chapter we assume that we fixed a particular triangle $\Delta\in X$ satisfying the conclusion of the Lemma 4.1 with some constant $a>0$.
\subsection{Subsequence complexity and a bootstrap}

In this subchapter we fix a triangle vertex and remind that a reduced quantity $Q_n$ counts only generalized diagonals emanating from it. As usually $P_n$ denotes the global (non-reduced) complexity for $\Delta$.

\begin{Theorem}[Bootstrap on subsequence complexity]
Assume that for some constant $\nu: 0<\nu\leq 1$ and all $n$ large enough $P_n<e^{n^{\nu}}$. Then for any $\mu: \gamma<\mu<\nu$,
$\liminf Q_{n}e^{-n^{\mu}}=0$, where $\gamma=\frac{-{\nu}^2+\sqrt{{\nu}^4 + 4{\nu}^2}}{2}$.

\end{Theorem}
\textbf{Remark.} One can easily check that $\gamma=\frac{-{\nu}^2+\sqrt{{\nu}^4 + 4{\nu}^2}}{2}<\nu$ for any $\nu>0$ so the assumption on $\mu$ makes sense.

\begin{proof}
We are going to prove the theorem by contradiction. In order to do it we assume that for some $\mu: \gamma<\mu<\nu$ and all $n$ large enough $Q_n>e^{n^{\mu}}$.
We first prove a technical lemma.
\begin{Lemma} For any numbers $\nu$, $\mu$, $\gamma$ satisfying the assumptions of theorem 4.2 there exists $\epsilon>0$ satisfying $\mu>\epsilon\nu$ and $(\epsilon+\nu)\mu>\nu$

\end{Lemma}
\begin{proof}
Consider the system:
\

$\begin{cases} \mu>\epsilon\nu \\  (\epsilon+\nu)\mu>\nu \end{cases}$
\

Considering the extreme case and substituting $\epsilon$ we obtain $(\frac{\mu}{\nu}+\nu)\mu=\nu$ or $(\mu+{\nu}^2)\mu={\nu}^2$. Solving this quadratic equation with $\nu$ as a parameter we
get that $\gamma=\frac{-{\nu}^2+\sqrt{{\nu}^4 + 4{\nu}^2}}{2}$ is a critical value for $\mu$ so for any $\mu>\gamma$ we may find $\epsilon$ satisfying the inequalities above.
\end{proof}

Let us introduce the notations $k(n)=n^{\nu}$, $c(n)=n^{\epsilon}$ with $\epsilon$ from lemma 4.2. When it does not lead to ambiguity we will use the notations $k(n)=k$, $c(n)=c$ in order not to overabuse the calculations.

\begin{Lemma} For any $n$ large enough there exists $s: n\leq s\leq n+c(k-1)$ such that $Q_{s+c}\geq 3Q_s$.

\end{Lemma}
\begin{proof}
Let us divide index interval $[n, n+kc]$ into $k$ subintervals of length $c$: $I_1=[n, n+c]$, $I_2=[n+c, n+2c]$, $\ldots$ , $I_k=[n+(k-1)c, n+kc]$ and prove that for one of the intervals $I_i, 1\leq i\leq k$: $Q_{n+ic}/Q_{n+(i-1)c}>3$. Assuming the opposite we have: $Q_{n+kc}/Q_n\leq 3^k$. On the other hand: $Q_{n+kc}/Q_n>e^{{(n+kc)}^{\mu}-n^{\nu}}$. Taking the logarithm of both inequalities we obtain:  ${(n+kc)}^{\mu}-n^{\nu}<k\ln(3)$. As $(\epsilon+\nu)\mu>\nu$ we get a contradiction for large enough $n$.
\end{proof}

We now need the following abstract lemma.

\begin{Lemma}
Assume there is a set $I$ of $n$ non-intersecting subintervals $I_1$, $I_2$,$\ldots$,$I_n$ of interval $[0, 1]$ satisfying $L^{-1}\leq\frac{|I_i|}{|I_j|}\leq L$ for some constant $L>0$ and a set $J$ of intervals $J_1$, $J_2$,$\ldots$,$J_n$ such that for each $i: 1\leq i\leq n$ intervals $J_i$ and $I_i$ have the same left endpoints and $|J_i|\leq m|I_i|$ for some $m>0$ such that $n\geq Lm$. Then $J$ contains a subset of $\frac{n}{Lm}$ non-intersecting intervals.
\end{Lemma}

\begin{proof} We assume that the set $I$ is naturally ordered by the left points of intervals $I_i$. Then $J_1=[x_1, x_2]$ contains not more than $mL$ intervals from $I$. We pick $J_1$ and then consider the reduced interval $[x_2, 1]$ with at least $n-mL$ intervals from $I$ left. We can repeat this procedure at least $\frac{n}{mL}$ times and so the proof is complete.
\end{proof}
\textbf{Remark.} Of course analogous statement holds if in lemma 4.4 assumptions $J_i$ and $I_i$ have the same \textsl{right} points.
\

\

Consider now the sequence of interval partitions $\xi_s$, $\xi_{s+1}$, $\ldots$, $\xi_{s+c}$. As $Q_{s+c}>3Q_s$ then by lemma 2.1 we can find the set $X_1$ of at least $Q_s-1$ intervals $[x_i, y_i]$ of $\xi_{s+c}$ with indices of endpoints in $[s+1, s+c]$. Then there exists at least $Q_s-1-\frac{n+kc}{b}$ intervals from $X$ with lengths smaller than $\frac{b}{n+kc}$, where $b$ is a constant from lemma 3.4. By pigeonhole principle at least half of these intervals have left indices larger than the right ones or vice versa. Without loss of generality we assume that $X_1$ contains the set $X_2$ of $\frac{1}{2}(Q_s-1-\frac{n+kc}{b})$ intervals $[x_i, y_i]$ such that the indices of $x_i$ are smaller than the indices of $y_i$ and with lengths smaller than $\frac{b}{n+kc}$ .
\

\

We divide the set $X_2$ on $a{(n+kc)}^2$ subsets by the interval length, where $a>0$ is a constant from the Lemma 4.1. Namely $Y_1=\lbrace I\in X_2 : e^{-1}<|I|\leq 1\rbrace$,
$Y_2=\lbrace I\in X_2 : e^{-2}<|I|\leq e^{-1}\rbrace$, $\ldots$ , $Y_{a(n+kc)^2}=\lbrace I\in X_2 : e^{-a(n+kc)^2}<|I|\leq e^{-a(n+kc)^2+1}\rbrace$. By pigeonhole principle
 there exists a set $X_3\subseteq X_2$ having $\frac{Q_s-1-\frac{n+kc}{b}}{2a{(n+kc)}^2}$ intervals such that for any pair of intervals $I_i, I_j\in X_3: |I_i|<e|I_j|$.
\

\

Any interval $I\in X_3$ corresponds to a pair of generalized diagonals with indices $p<q$ such that $s\leq p, q\leq s+c$. Let the points $P, Q$ represent the endpoints of unfoldings of these diagonals on the $xy$ - plane. By lemma 3.3 the segment $PQ$ is either an unfolding of a generalized diagonal of algebraic length bounded by $c$ or a triangle side.
For each point $P$ corresponding to to all the intervals $I\in X_3$ we consider the exit triangle in the $P$ -rose. Each such a triangle corresponds to a particular combinatorics of length at most $n+kc$. The kite angle of each exit triangle equals to $i2\alpha+j2\beta$, where $|i|+|j|\leq n+kc$. So there are less than ${(n+kc)}^2$ angular kite positions for such exit triangles. As there are also two positions of an exit triangle inside a corresponding kite we get that there are at most $2{(n+kc)}^2$ exit positions for exit triangles of $P$ -points for intervals $I\in X_3$. Applying pigeonhole principle again we get that there is a set of intervals $X_4\subseteq X_3$ such that all the exit triangles for the $P$ - points are in \textsl{ the same} exit position and moreover the cardinality $|X_4|=\frac{Q_s-1-\frac{n+kc}{b}}{4a{(n+kc)}^4}$
\

\unitlength .5mm 
\linethickness{0.4pt}
\ifx\plotpoint\undefined\newsavebox{\plotpoint}\fi 
\begin{picture}(253.5,109.107)(0,0)
\multiput(168.648,42.183)(.0875,-.058333){6}{\line(1,0){.0875}}
\put(103.433,54.588){\line(0,1){0}}
\put(63.092,29.591){\line(0,1){0}}
\thicklines
\put(23.247,15.732){\vector(1,0){87.611}}
\multiput(23.247,15.485)(.067335689,.0682120141){283}{\line(0,1){.0682120141}}
\multiput(42.303,34.789)(.0962014134,-.0673392226){283}{\line(1,0){.0962014134}}
\multiput(117.788,65.972)(.12375,.062){4}{\line(1,0){.12375}}
\put(109.373,18.702){$x$}
\multiput(23.494,15.732)(.15815730337,.06743398876){712}{\line(1,0){.15815730337}}
\put(100.216,36.026){\line(0,1){0}}
\put(100.216,36.026){\line(0,1){0}}
\put(18.605,15.893){$A$}
\multiput(135.875,63.768)(.128471,.059786){7}{\line(1,0){.128471}}
\multiput(137.674,64.605)(.128471,.059786){7}{\line(1,0){.128471}}
\multiput(139.473,65.442)(.128471,.059786){7}{\line(1,0){.128471}}
\multiput(141.271,66.279)(.128471,.059786){7}{\line(1,0){.128471}}
\multiput(143.07,67.116)(.128471,.059786){7}{\line(1,0){.128471}}
\multiput(144.868,67.953)(.128471,.059786){7}{\line(1,0){.128471}}
\multiput(146.667,68.79)(.128471,.059786){7}{\line(1,0){.128471}}
\multiput(148.466,69.627)(.128471,.059786){7}{\line(1,0){.128471}}
\multiput(150.264,70.464)(.128471,.059786){7}{\line(1,0){.128471}}
\multiput(152.063,71.301)(.128471,.059786){7}{\line(1,0){.128471}}
\multiput(153.861,72.138)(.128471,.059786){7}{\line(1,0){.128471}}
\multiput(155.66,72.975)(.128471,.059786){7}{\line(1,0){.128471}}
\multiput(157.459,73.812)(.128471,.059786){7}{\line(1,0){.128471}}
\multiput(159.257,74.649)(.128471,.059786){7}{\line(1,0){.128471}}
\multiput(161.056,75.486)(.128471,.059786){7}{\line(1,0){.128471}}
\multiput(162.854,76.323)(.128471,.059786){7}{\line(1,0){.128471}}
\multiput(164.653,77.16)(.128471,.059786){7}{\line(1,0){.128471}}
\multiput(166.451,77.997)(.128471,.059786){7}{\line(1,0){.128471}}
\multiput(168.25,78.834)(.128471,.059786){7}{\line(1,0){.128471}}
\multiput(170.049,79.671)(.128471,.059786){7}{\line(1,0){.128471}}
\multiput(171.847,80.508)(.128471,.059786){7}{\line(1,0){.128471}}
\multiput(173.646,81.345)(.128471,.059786){7}{\line(1,0){.128471}}
\multiput(175.444,82.182)(.128471,.059786){7}{\line(1,0){.128471}}
\multiput(177.243,83.019)(.128471,.059786){7}{\line(1,0){.128471}}
\multiput(179.042,83.856)(.128471,.059786){7}{\line(1,0){.128471}}
\multiput(180.84,84.693)(.128471,.059786){7}{\line(1,0){.128471}}
\multiput(182.639,85.53)(.128471,.059786){7}{\line(1,0){.128471}}
\multiput(184.437,86.367)(.128471,.059786){7}{\line(1,0){.128471}}
\multiput(186.236,87.204)(.128471,.059786){7}{\line(1,0){.128471}}
\multiput(188.035,88.041)(.128471,.059786){7}{\line(1,0){.128471}}
\multiput(189.833,88.878)(.128471,.059786){7}{\line(1,0){.128471}}
\multiput(191.632,89.715)(.128471,.059786){7}{\line(1,0){.128471}}
\multiput(193.43,90.552)(.128471,.059786){7}{\line(1,0){.128471}}
\multiput(195.229,91.389)(.128471,.059786){7}{\line(1,0){.128471}}
\multiput(197.028,92.226)(.128471,.059786){7}{\line(1,0){.128471}}
\multiput(198.826,93.063)(.128471,.059786){7}{\line(1,0){.128471}}
\multiput(200.625,93.9)(.128471,.059786){7}{\line(1,0){.128471}}
\multiput(202.423,94.737)(.128471,.059786){7}{\line(1,0){.128471}}
\multiput(204.222,95.574)(.128471,.059786){7}{\line(1,0){.128471}}
\multiput(206.021,96.411)(.128471,.059786){7}{\line(1,0){.128471}}
\multiput(207.819,97.248)(.128471,.059786){7}{\line(1,0){.128471}}
\multiput(209.618,98.085)(.128471,.059786){7}{\line(1,0){.128471}}
\multiput(211.416,98.922)(.128471,.059786){7}{\line(1,0){.128471}}
\multiput(213.215,99.759)(.128471,.059786){7}{\line(1,0){.128471}}
\multiput(215.013,100.596)(.128471,.059786){7}{\line(1,0){.128471}}
\multiput(216.812,101.433)(.128471,.059786){7}{\line(1,0){.128471}}
\multiput(218.611,102.27)(.128471,.059786){7}{\line(1,0){.128471}}
\multiput(220.409,103.107)(.128471,.059786){7}{\line(1,0){.128471}}
\multiput(222.208,103.944)(.128471,.059786){7}{\line(1,0){.128471}}
\multiput(224.006,104.781)(.128471,.059786){7}{\line(1,0){.128471}}
\multiput(225.805,105.618)(.128471,.059786){7}{\line(1,0){.128471}}
\multiput(227.604,106.455)(.128471,.059786){7}{\line(1,0){.128471}}
\multiput(229.402,107.292)(.128471,.059786){7}{\line(1,0){.128471}}
\multiput(231.201,108.129)(.128471,.059786){7}{\line(1,0){.128471}}
\multiput(23.545,15.557)(.20815649606,.06745374016){1016}{\line(1,0){.20815649606}}
\multiput(136.646,64.119)(.3350716724,.067443686){293}{\line(1,0){.3350716724}}
\multiput(234.612,83.88)(-.067233503,-.10564467){197}{\line(0,-1){.10564467}}
\multiput(221.367,63.068)(-.06688636,.51361364){88}{\line(0,1){.51361364}}
\multiput(215.481,108.266)(.0673309609,-.0860355872){281}{\line(0,-1){.0860355872}}
\multiput(234.471,83.739)(.17238,.05466){5}{\line(1,0){.17238}}
\multiput(236.195,84.286)(.17238,.05466){5}{\line(1,0){.17238}}
\multiput(237.919,84.833)(.17238,.05466){5}{\line(1,0){.17238}}
\multiput(239.643,85.379)(.17238,.05466){5}{\line(1,0){.17238}}
\multiput(241.367,85.926)(.17238,.05466){5}{\line(1,0){.17238}}
\multiput(137.067,64.119)(.067364641,-.153314917){181}{\line(0,-1){.153314917}}
\multiput(149.26,36.369)(-.1048189415,.0673426184){359}{\line(-1,0){.1048189415}}
\multiput(111.63,60.545)(.44257895,.06638596){57}{\line(1,0){.44257895}}
\multiput(136.857,64.329)(.0672389706,.0680147059){272}{\line(0,1){.0680147059}}
\multiput(155.146,82.829)(-.06688636,-.52318182){88}{\line(0,-1){.52318182}}
\put(132.232,67.482){$P$}
\put(235.873,75.891){$Q$}
\qbezier(82.408,34.687)(83.354,38.471)(80.516,39.733)
\qbezier(177.851,72.528)(179.322,77.363)(175.748,82.198)
\put(85.141,38.261){$\phi$}
\put(180.373,77.573){$\psi$}
\qbezier(78.835,15.557)(78.414,26.699)(73.789,37)
\put(79.886,21.653){$\xi$}
\put(156.197,85.772){$B$}
\put(151.993,33.846){$C$}
\put(109.807,15.416){\line(1,0){.9977}}
\put(111.803,15.416){\line(1,0){.9977}}
\put(113.798,15.416){\line(1,0){.9977}}
\put(115.794,15.416){\line(1,0){.9977}}
\put(117.789,15.416){\line(1,0){.9977}}
\put(119.785,15.416){\line(1,0){.9977}}
\put(121.78,15.416){\line(1,0){.9977}}
\put(123.776,15.416){\line(1,0){.9977}}
\put(125.771,15.416){\line(1,0){.9977}}
\put(127.767,15.416){\line(1,0){.9977}}
\put(129.762,15.416){\line(1,0){.9977}}
\put(131.757,15.416){\line(1,0){.9977}}
\put(133.753,15.416){\line(1,0){.9977}}
\put(135.748,15.416){\line(1,0){.9977}}
\put(137.744,15.416){\line(1,0){.9977}}
\put(139.739,15.416){\line(1,0){.9977}}
\put(141.735,15.416){\line(1,0){.9977}}
\put(143.73,15.416){\line(1,0){.9977}}
\put(145.726,15.416){\line(1,0){.9977}}
\put(147.721,15.416){\line(1,0){.9977}}
\put(149.717,15.416){\line(1,0){.9977}}
\put(151.712,15.416){\line(1,0){.9977}}
\put(153.708,15.416){\line(1,0){.9977}}
\put(155.703,15.416){\line(1,0){.9977}}
\put(157.698,15.416){\line(1,0){.9977}}
\put(159.694,15.416){\line(1,0){.9977}}
\put(161.689,15.416){\line(1,0){.9977}}
\put(163.685,15.416){\line(1,0){.9977}}
\put(165.68,15.416){\line(1,0){.9977}}
\put(167.676,15.416){\line(1,0){.9977}}
\put(169.671,15.416){\line(1,0){.9977}}
\put(171.667,15.416){\line(1,0){.9977}}
\put(136.5,64){\line(1,0){117}}
\multiput(149.109,36.359)(.065637,-.122587){7}{\line(0,-1){.122587}}
\multiput(150.028,34.643)(.065637,-.122587){7}{\line(0,-1){.122587}}
\multiput(150.947,32.927)(.065637,-.122587){7}{\line(0,-1){.122587}}
\multiput(151.866,31.211)(.065637,-.122587){7}{\line(0,-1){.122587}}
\multiput(152.785,29.495)(.065637,-.122587){7}{\line(0,-1){.122587}}
\multiput(153.704,27.778)(.065637,-.122587){7}{\line(0,-1){.122587}}
\multiput(154.623,26.062)(.065637,-.122587){7}{\line(0,-1){.122587}}
\multiput(155.542,24.346)(.065637,-.122587){7}{\line(0,-1){.122587}}
\multiput(156.461,22.63)(.065637,-.122587){7}{\line(0,-1){.122587}}
\multiput(157.38,20.913)(.065637,-.122587){7}{\line(0,-1){.122587}}
\multiput(158.299,19.197)(.065637,-.122587){7}{\line(0,-1){.122587}}
\multiput(159.218,17.481)(.065637,-.122587){7}{\line(0,-1){.122587}}
\multiput(160.136,15.765)(.065637,-.122587){7}{\line(0,-1){.122587}}
\multiput(161.055,14.049)(.065637,-.122587){7}{\line(0,-1){.122587}}
\multiput(161.974,12.332)(.065637,-.122587){7}{\line(0,-1){.122587}}
\multiput(162.893,10.616)(.065637,-.122587){7}{\line(0,-1){.122587}}
\multiput(163.812,8.9)(.065637,-.122587){7}{\line(0,-1){.122587}}
\multiput(164.731,7.184)(.065637,-.122587){7}{\line(0,-1){.122587}}
\multiput(165.65,5.468)(.065637,-.122587){7}{\line(0,-1){.122587}}
\put(177.25,37.75){\vector(-2,-3){.141}}\qbezier(187.5,63.75)(182.125,33)(157.25,21.25)
\put(182.5,31.5){$\theta$}
\end{picture}
\

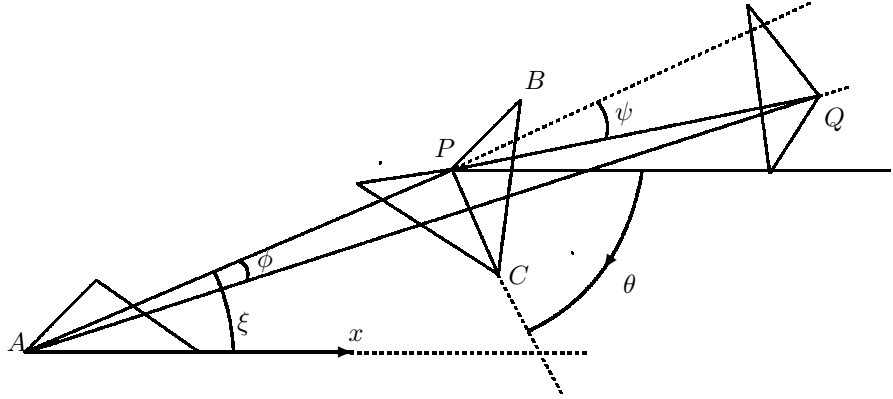
Fig 10.
\

\

Now we take a closer look at the set of intervals $X_4$. For a moment we fix a particular interval $I_i\in X_4$, where $1\leq i\leq|X_4|$ and look at the unfolding sequences, corresponding to the endpoints of $I_i$.
As usually we denote the corresponding points on the $xy$ - plane as $P$ and $Q$. From fig 10. we see that $\angle CPQ=\xi-\psi-\theta$.
We remind that for all intervals $I\in X_4$ the exit triangles $PBC$ are in the same exit position and so the oriented exit angle $\theta$ is the same for all intervals under consideration.
The segment $PQ$ is an unfolding of a generalized diagonal which is uniquely characterized by $\angle CPQ$. In particular any two such segments for different intervals $I$ correspond to different generalized diagonals if the corresponding angles $\angle CPQ$ are different.
\

\

Consider now the triangle $APQ$. By lemma 3.4 there is a constant $K>0$ such that $\psi<K\phi(n+kc)$. Now we consider $X_4$ as the set of non-intersecting subintervals of $[0, \pi]$.
In this interpretation for any such a subinterval $I=[x, y]$ we have $y=\xi$ and $\phi=|I|$. Now we are exactly in the position to apply lemma 4.3 with $L=e$ and $m=K(n+kc)$ and where intervals $J_i$ have the same \textsl{right} endpoints with corresponding intervals $I_i$. As $\psi<K(n+kc)\phi$ then in terms of lemma 4.3 the corresponding angle $\xi-\psi$ belongs exactly to the interval $J_i$. By lemma 4.3 we can find at least $\frac{|X_4|}{eK(n+kc)}$ non-intersecting intervals $J_i$.
\

\

Summarizing our observations we see that we may the set $X_5\subseteq X_4$ of cardinality $|X_5|=\frac{Q_s-1-\frac{n+kc}{b}}{4aeK{(n+kc)}^5}$ such that the corresponding connecting segments $PQ$ correspond to \textbf{different} generalized diagonals of algebraic length bounded by $c$. But this in turn implies that $|X_5|\leq P_c$. By theorem 4.2 assumptions
$Q_s\geq Q_n>e^{n^{\mu}}$ and $P_c\leq e^{c^{\nu}}=e^{n^{\epsilon\nu}}$ which gives: $\frac{e^{n^{\mu}}-1-\frac{n+kc}{b}}{4aeK{(n+kc)}^5}\leq e^{n^{\epsilon\nu}}$. However as $\mu>\epsilon\nu$ we get a contradiction by taking $n$ large enough.
\end{proof}
\subsection{Gap estimates and global complexity bootstrap}
In this chapter we still assume that for some $\nu>0$ and large $n$ we have $P_n<e^{n^{\nu}}$.  As a result of the theorem 4.2 for any  $\mu>\gamma$ there exists an increasing sequence of times $n_i$ uniquely characterized by the property $Q_{n_i}<e^{n_i^{\mu}}$. We are now going to estimate the gaps $n_{i+1}-n_i$.

\begin{Theorem} Let $\nu>0$, $\mu>\gamma=\frac{-{\nu}^2+\sqrt{{\nu}^4 + 4{\nu}^2}}{2}$ and $n_i$ is a sequence characterized by the property $Q_{n_i}<e^{n_i^{\mu}}$. Then for any $\epsilon>0$ and for any $i$ large enough $n_{i+1}-n_i<n_i^{1+\epsilon}$.

\end{Theorem}

\begin{proof} We fix small enough $\epsilon>0$ and introduce the notations $k(n)=n^{\mu}$, $c(n)=n^{1-\mu+\epsilon}$. When it does not lead to ambiguity we will for brevity use the notations $k=k(n_i)$ and $c=c(n_i)$. We are going to prove the theorem by contradiction so let us now assume that there exists a subsequence of $n_i$ satisfying $n_{i+1}-n_i>n_i^{1+\epsilon}$. We still denote this subsequence by $n_i$ in order not to overabuse the notations.

\begin{Lemma} For any $i$ large enough there exists $s: n_i\leq s\leq n_i+kc$ such that $Q_{s+c}>3Q_s$.

\end{Lemma}

\begin{proof}We divide index interval $[n_i, n_i+kc]$ into $k$ subintervals of length $c$: $I_1=[n_i, n_i+c]$, $I_2=[n_i+c, n_i+2c]$, $\ldots$, $[n_i+(k-1)c, n_i+kc]$. We would like to prove that for one of the intervals $I_j, 1\leq j\leq k$: $Q_{n_i+jc}/Q_{n_i+(j-1)c}>3$. Assuming the opposite we have: $Q_{n_i+kc}/Q_{n_i}\leq 3^k$. On the other hand from theorem 4.3 assumptions we have: $Q_{n_i+kc}/Q_{n_i}>e^{{(n_i+kc)}^{\mu}-n_i^{\mu}}$. Taking the logarithm of both inequalities we obtain:  ${(n_i+kc)}^{\mu}-n_i^{\mu}<k\ln(3)$. As $(\mu+(1-\mu+\epsilon))\mu>\mu$ we get a contradiction for large enough $n_i$.
\end{proof}

At this moment we entirely repeat all the arguments in the proof of the theorem 4.2 only changing $n$ to $n_i$, and $k(n)$, $c(n)$ in the chapter 4.1 to the newly introduced $k(n_i)$, $c(n_i)$. We do not write down all these arguments because literally nothing changes and all the arguments repeat word by word until we get the set of generalized diagonals $X_5$ of cardinality $|X_5|=\frac{Q_s-1-\frac{n_i+kc}{b}}{4aeK{(n_i+kc)}^5}$ which must satisfy: $|X_5|<P_c$.
\

\

As $Q_s\geq Q_{n_i}\geq \frac{Q_{n_i+1}}{2}\geq\frac{e^{{(n_i+1)}^{\mu}}}{2}$ and $P_c\leq e^{c^{\nu}}=e^{n_i^{(1-\mu+\epsilon)\nu}}$ we have a contradiction for large enough $n_i$ if
$\mu>(1-\mu+\epsilon)\nu$. It is clear that the last condition holds for any small enough $\epsilon>0$ if $\mu>(1-\mu)\nu$ or $\mu>\frac{\nu}{1+\nu}$ which immediately follows from the easily checked inequality $\mu>\gamma=\frac{-{\nu}^2+\sqrt{{\nu}^4 + 4{\nu}^2}}{2}>\frac{\nu}{1+\nu}$.
\end{proof}

Now we combine two previous theorems and get the global complexity bootstrap.

\begin{Theorem} Assume that for some constant $\nu: 0<\nu\leq 1$ and all $n$ large enough: $P_n<e^{n^{\nu}}$. Then for any $\mu>\gamma=\frac{-{\nu}^2+\sqrt{{\nu}^4 + 4{\nu}^2}}{2}$ and all $n$ large enough $P_{n}<e^{n^{\mu}}$.
\end{Theorem}
\begin{proof} It is enough to prove that $Q_{n}<e^{n^{\mu}}$ for all $n$ large enough. Pick any $\epsilon>0$ small enough and then pick positive $\delta\ll\epsilon$ small enough. Consider a sequence $n_i$ corresponding to $\mu=\gamma+\delta$. By theorem 4.3 for all $i$ large enough: $n_{i+1}<n_{i}+n_{i}^{1+\delta}$. As we are able to slightly perturb $\delta$ if needed we may assume: $n_{i+1}<n_i^{1+\delta}$.
\

\

Now we for large $n$ take $i$ such that $n_i\leq n<n_{i+1}$. By monotonicity $Q_{n_i}\leq Q_n\leq Q_{n_{i+1}}$, so we get:
$Q_n\leq e^{n_{i+1}^{\gamma+\delta}}\leq e^{n_i^{(1+\delta)(\gamma+\delta)}}\leq e^{n^{(1+\delta)(\gamma+\delta)}}\leq e^{n^{\gamma+\epsilon}}$.
As $\epsilon$ can be chosen arbitrarily small the proof is complete.
\end{proof}

One can easily see that the bootstrap function $\gamma=f(\nu)=\frac{-{\nu}^2+\sqrt{{\nu}^4 + 4{\nu}^2}}{2}$ is a monotone function satisfying $0<f(\nu)<\nu$ for all $\nu>0$ and so the iterations $f^k(\nu)$ converge to $0$. In particular for any $\epsilon>0$ there exists $k$ such that $f^k(1)<\epsilon$. As for any triangle and any $n$ large enough $P_n<e^{n^1}$ then the $k$ times application of the theorem 4.4 implies our main result:

\begin{Theorem}[Weakly exponential estimate] For a typical triangle and any $\epsilon>0$ there is a constant $C>0$ such that: $P_n<Ce^{n^{\epsilon}}$.
\end{Theorem}


\newpage
\begin{thebibliography}{99}


\bibitem{BF}J. Cassaigne, P. Hubert, S. Troubetzkoy, Complexity and growth for polygonal billiards, Ann. Inst. Fourier, 2002, 52:3, pp. 835-847.
\bibitem{BF}E. Gutkin, M. Rams,  Growth rates for geometric complexities and counting functions in polygonal billiards, ETDS, 2009, 29, pp. 1163-1183.
\bibitem{BF}E. Gutkin, S. Tabachnikov, Complexity of piecewise convex transformations in two dimensions, with applications to polygonal billiards on surfaces of constant curvature,
MMJ, 2006, 6:4, pp. 701-772.
\bibitem{BF}E. Gutkin, S. Troubetzkoy, Directional flows and strong recurrence for polygonal billiards, Proceedings of the International Congress of Dynamical Systems, Montevideo, Uruguay.
\bibitem{BF}V. Kaloshin, I. Rodnianski, Diophantine properties of elements of SO(3), GAFA, 2001, 11, 953-970.
\bibitem{BF}A. Katok, Five most resistant problems in dynamics, www.math.psu.edu/katok a/pub/5problems-expanded.pdf
\bibitem{BF}A. Katok, The growth rate for the number of singular and periodic orbits for a polygonal billiard, Comm. Math. Phys., 1987, 111, 151-160
\bibitem{BF}A. Katok, A. Zemlyakov,  Topological transitivity of billiards in polygons, Mat. Zametki, 1975, 18:2, pp. 291-300
\bibitem{BF}H. Masur, The growth rate of trajectories of a quadratic differential, ETDS, 1990, 10, pp. 151-176.
\bibitem{BF}H. Masur, Lower bounds for the number of saddle connections and closed trajectories of a quadratic
differential, Holomorphic functions and moduli, 1, D. Drasin, Springer-Verlag, 1988.
\bibitem{BF}D. Scheglov, Growth of periodic orbits and generalized diagonals for typical triangular billiards, to appear in JMD.
\end{thebibliography}
\end{document}